\RequirePackage{fix-cm}
\documentclass[smallextended]{svjour3}       
\smartqed  
\usepackage{graphicx}

\usepackage{amsfonts, latexsym, amssymb}

\usepackage[fleqn]{amsmath}
\usepackage{graphicx}
\usepackage{graphics}
\usepackage[comma, numbers]{natbib} 
\usepackage[comma]{natbib} 
\usepackage{color}
\usepackage{comment}
\usepackage{algorithm}
\usepackage{algorithmic}
\usepackage{lscape}
\usepackage{footnote}
\usepackage{algorithm,caption}

\usepackage{epsfig}
\usepackage{fancybox}
\usepackage{psfrag}
\usepackage{tabularx}
\usepackage{float}
\usepackage{multirow}
\usepackage{rotating}
\usepackage{cancel}
\usepackage{epstopdf}
\usepackage{booktabs}

\usepackage{subcaption}

\usepackage{pgfplots}
\pgfplotsset{compat=1.8}
\usepgfplotslibrary{statistics}

\date{}

\newcommand{\om}{\omega}

\newtheorem{thm}{THEOREM}[section]

\newtheorem{prop}[thm]{PROPOSITION}

\newlength{\myheig}
\newenvironment{hanglist} 
{\begin{list}{}{\setlength{\itemsep}{4pt}
			\parindent 0pt  \setlength{\parsep}{0pt}\setlength{\leftmargin}{+25pt}
			\setlength{\itemindent}{-\parindent}}}{\end{list}}

\begin{document}

\title{UV mission planning under uncertainty in vehicles' availability
}


\author{Saravanan Venkatachalam      \and
        Jonathon M. Smereka 
}


\institute{S. Venkatachalam \at
             Department of Industrial and Systems Engineering, \\
             Wayne State University, 
             4815 Fourth Street, Detroit, Michigan. \\
             Tel.: +1-313-577-1821\\
             \email{saravanan.v@wayne.edu}           
           \and
           J. M. Smereka \at
           Ground Vehicle Robotics (GVR) team, \\
           U.S. Army CCDC Ground Vehicle Systems Center (GVSC), \\
           6305 E Eleven Mile Rd, Warren, Michigan.\\
           \email{jonathon.m.smereka.civ@mail.mil}           
}

\date{Received: date / Accepted: date}

\maketitle

\begin{abstract}
Heterogeneous unmanned vehicles (UVs) are used in various defense and civil applications. Some of the civil applications of UVs for gathering data and monitoring include civil infrastructure management, agriculture, public safety, law enforcement, disaster relief, and transportation. This paper presents a two-stage stochastic model for a fuel-constrained UV mission planning problem with multiple refueling stations under uncertainty in availability of UVs. Given a set of points of interests (POI), a set of refueling stations for UVs, and a base station where the UVs are stationed and their availability is random, the objective is to determine route for each UV starting and terminating at the base station such that overall incentives collected by visiting POIs is maximized. We present an outer approximation based decomposition algorithm to solve large instances, and perform extensive computational experiments using random instances. Additionally, a data driven simulation study is performed using robot operating system (ROS) framework to corroborate the use of the stochastic programming approach.  

\end{abstract}

\section{Introduction}
\indent Advances in wireless networks, sensing, and robotics have led to various applications for Unmanned Vehicles (UVs). Crop monitoring \cite{Thomasson_2001,Willers_2005,Willers_2008,Willers_2009a}, forest fire monitoring \cite{casbeer}, ecosystem management \cite{pollutant, wildfire}, ocean bathymetry \cite{Ferreira}  are some of the environmental sensing applications. Similarly, disaster management \cite{disastermanagement} and  border surveillance \cite{searchandrescue, Krishna2012cdc} are some of the civil security applications. UVs are frequently used by these applications to collect data such as visible/infra-red/thermal images, videos of specified points of interests (POIs) or designated waypoints, and environmental data such as temperature, moisture, humidity using onboard sensors, and deliver them to a base station. For military operations, combat zones pose important challenges \cite{chgral,zaloga2011unmanned,krishnamoorthy2012uav}, hence the missions like  intelligence, surveillance and reconnaissance (ISR) are critical. Very effective and efficient information collection mechanisms are required for successful ISR missions, and UVs are an important asset for ISR. Additionally, UVs play vital role in search and rescue, target engagements, environmental mapping, disaster area surveying, and mapping and convoy operations for resupply missions. Also, UVs are preferred over other collection resources in instances like unsuitable terrain, harsh and hostile environment, and also in tedious information collection processes.

\indent Even though there are a lot of advantages in using UVs, they have limited payload capacity and distance range. This requires the UVs to make multiple stops at the refueling or recharging stations before they can complete their entire mission. The stops are also required for other purposes like security halt or mission handover. Also, in some applications, the total allotted time for a mission is also limited. This study focuses on fundamental questions related to changes in the availability of UVs during the course of a mission. For example, based on the information collected by the UVs at the POI sites, surveillance mission objectives may change or new tasks may be added to a sub-group of UVs or sensors. The mission adjustments may, for instance, result in attempting to assign an asset that is not currently available. If there is a chance that an UV is not available after a certain time window, it may not be beneficial to assign tasks for an UV that requires much longer to travel than the time window. Additional challenges include changes in expected terrain, obstacles which restrict movement, or asset failures, and all of these can result in uncertainty in the availability of UVs. An example in a military application is that when a ground UV is used in a hostile terrain with improvised explosive devices (IEDs), and conducting anti-IED sweeps or explosive ordinance disposal can lead to compromises or casualties in UV assets. Similarly, sensor failures or terrain incompatibilities for the UVs may also lead to their unavailability during a mission.\par  

\indent A mission routing plan provides a high-level detail on the sequence of POIs and refueling stations it has to visit. Due to their limited range and uncertainties in their availability during a mission, it is critical to formulate and efficiently solve the mission routing problem to successfully harness the benefits of the UVs. These problems are combinatorial in nature, and NP-hard problems such as multiple traveling salesman problem (TSP) and distance or capacity constrained vehicle routing problem are special cases of the UV mission planning problem described in this work. The scope of this paper is to develop high-level mission planning algorithm for the UVs with uncertainty in their availability during the mission. These problems are also referred to as routing problems \cite{toth2002vehicle} in the literature. Since they are hard to solve, and NP-hard in general, mission planning problems are solved offline before the start of a mission. Using the high-level routing plan for each UV, the low-level \textit{`path planning algorithms'} solve the challenge of finding an optimal trajectory between a pair of source and destination while considering obstacles, and provide closed-loop control signals to each UV so that they can follow the trajectory with minimum deviations \cite{de2020global, chen2015path, xia2020cooperative}. Due to trackability reasons, the complexities in low-level planning are not considered in the high-level planning. Potential field methods, $A^{*}$ search algorithm, and Rapidly-exploring Random Tree (RRT) \cite{latombe2012robot} are a few of the commonly used low-level online path planning algorithms. The readers are referred to \cite{yu2015sense} for a review of other state-of-the-art low-level path planning algorithms. 

From this brief description, the following four points summarize the main contributions of this work: (i) a risk-neutral two-stage stochastic programming model to obtain the routes for each of the UVs with fuel constraints and uncertainties in the availability of UVs; (ii) a reformulation for the two-stage stochastic model; (iii) an exact method using branch-and-cut procedure to solve the two-stage stochastic model to optimality; and (iv) extensive computational experiments using random instances, and simulation studies using robot operating system (ROS) as a middleware to corroborate the efficiency of the proposed approach, both quantitatively and qualitatively.

The remainder of the paper is presented as follows: description of the problem and literature review are presented in Section \ref{prob}; the details of two-stage stochastic programming model are given in Section \ref{model}; a branch and cut decomposition algorithm for solving the instances of two-stage  stochastic programming model is proposed in Section \ref{alog}; computational experiments for the proposed algorithm are performed using random instances in section \ref{comp}, and furthermore, perform data driven simulation study to corroborate the significance of stochastic model compared to its deterministic counterpart using a simulation package in robot operating system (ROS); finally, conclusions are presented in Section \ref{con}.

\section{Problem Description and Literature Review}\label{prob}

\indent The fundamental objective is to assign POIs to the UVs while incorporating possible changes in their availability of some resources in the future that could impact the ability of the team to complete the mission. These problems which can be formulated as two-stage stochastic mission planning problems are new and have not been considered in the literature. The problem is stated as follows: given a team of UVs and a subset of the team is randomly available for the mission, and a set of POI sites to visit, find a mission for each UV such that POIs are visited within a given duration, and an objective  based on the incentives of POIs visited by the UVs is maximized. For example, a typical and useful objective to address is the information collected along the traversed missions. In the absence of changes in the input data or resources, the problems considered in this research are already NP-Hard and computationally challenging to solve. In the presence of uncertainties, solving these problems require developing novel computational tools in an interdisciplinary area of research spanning combinatorial and stochastic optimization. Fig. \ref{stochastic} illustrates the effect of considering uncertainty in the availability of UVs. Fig. \ref{stochastic} (a) denotes the optimal solution for a deterministic UV mission planning problem which is sub-optimal for the UVs while considering uncertainty for the availability of UVs. Fig. \ref{stochastic} (b-c) are the optimal solutions for stochastic UV mission-planning instances having different chances of availability for UV2. As the figures denote, when the chances of the availability of UV2 decreases, the number of assigned POIs to UV2 also decreases. \vspace{\myheig}

\begin{figure}[h!]
	\centering
	\includegraphics[scale=0.40]{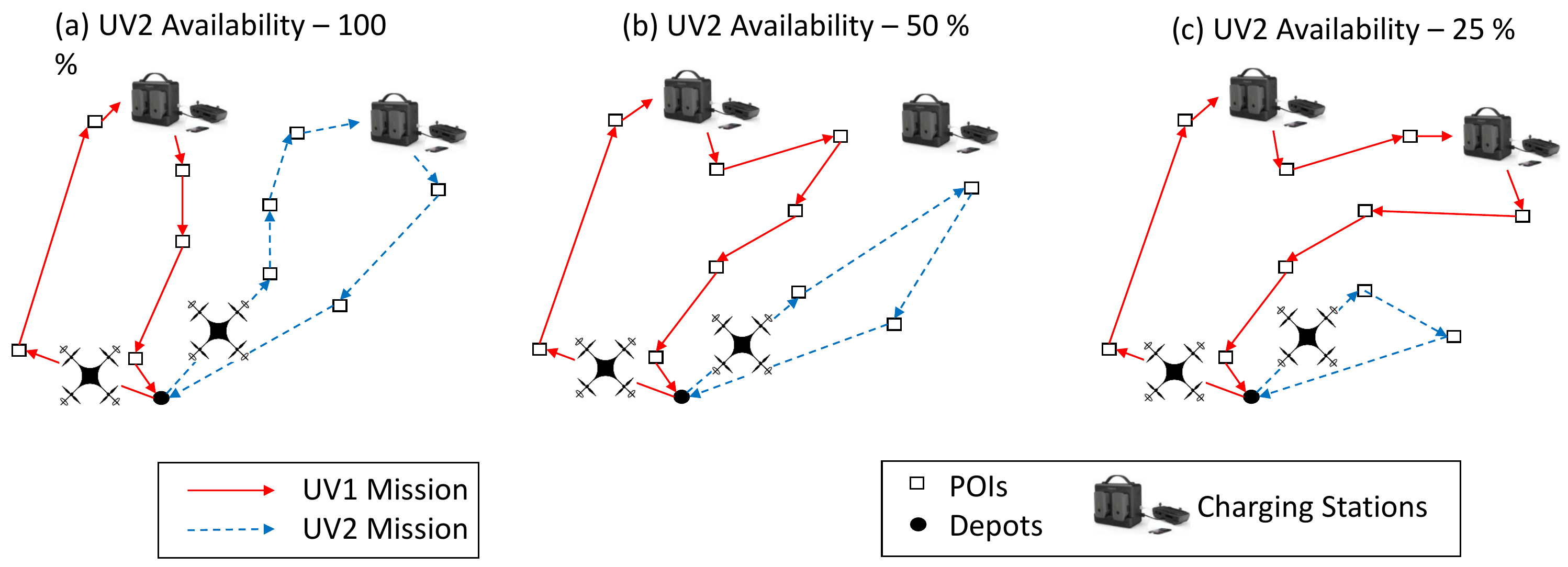}
	\caption{An illustration for considering uncertainty in the availability of UVs with fuel-constraints}
	\label{stochastic}
\end{figure}

\indent A survey on motion planning techniques for UVs under uncertainty is given in 
\cite{dadkhah2012survey}. We refer to the basic problem considered in this study as the UV mission planning problem with stochastic vehicle availability (SVA). A variation of deterministic SVA version reduces to a generalization of an orienteering problem (OP) which possesses the characteristics of both the well-known traveling salesman problem (TSP) and knapsack problem (KP). OP and TSP are NP-Hard \cite{golden1987orienteering, Applegate:2007,Held1970}. OP has received a good attention in the literature and a survey on OP is given in \cite{vansteenwegen2011orienteering}. Though there are a lot of practical relevance for stochastic OP, literature is very limited. The authors in \cite{tang2005algorithms} present a chance constrained stochastic program with an objective to maximize rewards collected by visiting POIs, and the probabilistic constraints are used  for the total actual mission duration. The authors have presented exact and heuristic methods. A robust optimization model with uncertainty in fuel usage between POIs is presented in \cite{evers2014robust}. The author presented different types of uncertainty sets for fuel usages, and also demonstrated the agility in solutions of robust OP. In \cite{ilhan2008orienteering}, OP with stochastic profits is considered and a mission is determined within a prescribed time limit. OP with stochastic travel and service times is presented in \cite{campbell2011orienteering}. An integer L-shaped algorithms for OP with uncertainty in travel and service time is presented in \cite{teng2004integer}. The model considers a soft capacity constraint so a penalty is paid for exceeding a given time limit. The authors in \cite{evers2014two} present a two-stage stochastic model for OP with uncertainty in stochastic weights, and sample average approximation with a heuristic is used as a methodology. The problem setup is single depot with no refueling option. Accordingly, to the best of our knowledge, while there have been several approaches to influence resulting mission based on asset availability given some set of constraints when visiting POIs, there is no multi-depot refueling OP problems with uncertainty in availability of vehicles directly addressed as part of the optimization. That is, in this work we are expanding the space of OP solutions when using multiple vehicles, multiple depots, and uncertain availability of UVs.

When the availability of resources such as vehicles or sensors change, the resulting SVA can be posed as a two-stage or a multi-stage stochastic optimization problem. In a two-stage stochastic program, the first-stage decisions are made with deterministic parameters before the realization of random variables representing uncertainty are revealed. Once the random events occur, the recourse decisions are made in the second-stage such that the missions decided for UVs in the first-stage have minimum conflict due to the new information from the random events. Thus, the first-stage decisions are to construct mission for each UV without revealing the randomness in their availability. Thereafter, based on the first-stage UV missions and realization of uncertainties, the second-stage decisions propose the recourse actions to adjust the profits attained in the first-stage.

\section{Notation and Model Formulation}\label{model}
In this section, we introduce the two-stage model for mission planning problems with the uncertainty in availability of UVs. POIs for a mission are defined in a set $P$ with their elements denoted as $\{p_1,\dots,p_n\}$, and similarly, let $R$ be the set of refueling or recharging stations with the elements $\{r_0, r_1,\dots, r_k\}$, and $r_0$ is the base station where all the UVs are initially stationed, and also assumed that UVs cannot recharge at $r_0$. The set $\bar{R}$ represents set $R$ without base station $r_0$. There are `$M$' UVs where each UV is indexed as $`m$', and all the `$M$' UVs are fueled to their capacity. The mathematical model for SVA is defined on a directed graph $G=(V,E)$ where $V$ is the set of POIs and recharging stations, $V=P\cup R$, and $E$ denotes the set of edges connecting a pair of nodes in $V$, and there are no self-loops in $G$. The set $P$ is indexed by $j$, and let $e_j^m$ represents an incentive for an UV `$m$' for visiting $j^{th}$ POI. The UVs are considered as heterogeneous as the incentive is UV dependent. For every edge $(i,j)$ in the set $E$, let $f_{ij}$ represents the fuel or time consumed by an UV to traverse it. Also, we assume that triangle inequalities are preserved, i.e., for every set of edges $i,j,k \in V$, $f_{ij} + f_{jk} \geq f_{ik}$. Another parameter is $F_m$ denoting the maximum distance an UV `$m$' can travel before it should be refueled. Similarly, let $\delta_m$ represents the maximum distance or time allowed for an UV during the mission. Finally, $q$ is a large constant used in the first-stage of the formulation which is usually set to the number of POIs, i.e., $q=|P|$. The objective of the mission planning problem is to determine the routes for each UV such that they never run out of fuel, and some of their availability are uncertain while maximizing the incentives collected by all UVs. A random variable $\tilde{\omega}$ is used to represent the uncertainty in the availability of UVs, and let $\omega$ represents a realization for $\tilde{\omega}$. The only  parameter used in the second-stage of the formulation is $\alpha_{m}(\om)$ which takes a value of $1$ or $0$ denoting the availability of the $m^{th}$ UV  for the scenario $\om$. A summary of notation is presented in Table \ref{tab:notations}. The two-stage stochastic programming model is presented in the next section.

\subsection{Risk-neutral Two-stage Stochastic Programming Recourse Model}\label{sec-model}
Considering the random variable $\tilde{\omega}$ to represent uncertainty for SVA, an abstract probability space is denoted by $(\Omega,\mathcal{F},\Xi)$, where $\Omega$ is the sample space, $\mathcal{F}$ is a $\sigma$-algebra on $\Omega$, and $\Xi$ is a probability measure on $\Omega$. We consider the case of a finite probability space, where $\Omega=\{\omega^1,\omega^2,\ldots,\omega^N\}$ with corresponding probabilities $\rho(\omega^1),\rho(\omega^2),\ldots,\rho(\omega^N)$ and $\sum_{\omega \in \Omega}\rho(\omega)=1$. In a risk-neutral formulation, $\mathbb{E}_{\tilde{\omega}}$ is an expectation operator and for a random variable function $\phi$, $\mathbb{E}_{\tilde\omega}(\phi) = \sum_{\omega \in \Omega}\rho(\omega)\phi_\omega$, where $\phi_\omega$ is a realization. An $\omega$ represents  random event for the availability of UVs with a probability $\rho(\omega)$. In a two-stage recourse model, a decision must be made here-and-now (first-stage) before future uncertainties are realized. The second-stage recourse problem for each realization `$\omega$' is solved for the given first-stage solution, and it's objective function value is weighted by its corresponding probability of occurrence $\rho(\omega)$. An optimal solution is attained when we reach the best objective value for the first- and second-stage objective functions as a whole.  

Next, we present a two-stage stochastic programming model based on standard routing problems. The first-stage decision variables used in the formulation are as follows: $y_{ij}^m$ is a binary variable taking a value `0' or `1' representing whether the edge is traversed by the UV $`m$' or not, respectively; $x_{ij}^m$ is a continuous variable representing the total distance traveled by an UV `$m$' when it reaches a POI `$j$' from a refueling or base station where $(i,j) \in E$; $z_r^m$ is a binary variable taking a value `0' or `1' based on whether the refueling station $r$ is used by an UV `$m$' or not, and this variable is used only for the refueling stations, i.e., $r \in R\setminus\{r_0\}$. The notation used to represent the set of in-degree and out-degree edges are as follows: for any set $V'\subset V$, $\beta^+(V') = \{(i, j) \in E: i \in V', j \notin V'\}$ where for any $E' \subseteq E$, $y(E') = \sum_{(i,j)\in E'} y_{ij}^m$.

\begin{table}
	\centering
	\begin{tabular}{ll}
		\toprule 
		\bf  Symbol &  \bf  Description \\ 
		\midrule 
		$P = \{p_1, \dots, p_n\}$ & set of $n$ POIs  \\
		$R = \{r_0, \dots, r_k\}$ & set of refueling or recharging stations, $r_0$ is the base station \\ 
		$F_m$ & fuel capacity of $m^{th}$ UV \\ 
		$G = (V, E)$ & directed graph with $V = P\cup R$  \\ 
		$e_{j}^m$ & incentive collected at POI $j$ by $m^{th}$ UV \\
		$\delta_m$ & maximum allowed distance for $m^{th}$ UV \\				
		$f_{ij}$ & fuel consumed for an edge $(i,j) \in E$ \\
		$\Omega$ & set of scenarios, $\omega \in \Omega$ is realization of random variable   \\
		$\rho(\omega)$ & probability of occurrence for $\omega$ \\ 
		$\alpha_{m}(\om)$ & availability of an UV $m$ for the scenario $\om$ \\
		\bottomrule
	\end{tabular}
	\caption{Notation}
	\label{tab:notations}
\end{table}

Using the above variables, the two-stage stochastic programming formulation `SVA-TS' is given as follows:

\begin{flalign}
& \text{Maximize} \quad \text{SVA-TS:} \sum_{(i,j)\in E, m \in M} e_{j}^m y_{ij}^m + \mathbb{E}_{\tilde{\omega}} \phi^{ec}(x,y,\omega) \notag& \\
&\text{Subject to:} \notag & \\
&\sum_{i\in V} y_{ji}^m = \sum_{i\in V} y_{ij}^m \quad \forall \, j\in V\setminus\{r_0\}, m \in M,\label{eq:1}& \\
&\sum_{i\in V, m \in M} y_{ir_0}^m = \left\vert{M}\right\vert \text{ and } \sum_{i\in V, m \in M} y_{r_0i}^m = \left\vert{M}\right\vert, \label{eq:2} &\\
&y(\beta^+(S))^m \geq z^m_r \quad \forall \, r\in S\cap R, S\subset V\setminus\{r_0\}:S\cap R \neq \emptyset, m \in M, \label{eq:3} &\\ 
& \sum_{i \in V}y_{ri}^{m} \leq q. z_{r}^{m} \hspace{1.65cm} \forall   r \in \bar{R},m \in M,  &&   \label{eq:4} \\ 
&\sum_{i\in V, m \in M} y_{ij}^m \leq 1 \text{ and } \sum_{i\in V, m \in M} y_{ji}^m \leq 1 \quad \forall \, j \in P, \label{eq:5} &\\
& \sum_{j\in V}x_{ij}^m - \sum_{j\in V}x_{ji}^m  = \sum_{j\in V}f_{ij}y_{ij}^m \quad \forall i \in P, m \in M, \label{eq:7} & \\
&x_{ri}^m = f_{ri}y_{di}^m \quad \forall \, i\in V, \, r \in R, m \in M, \label{eq:11} & \\
&x_{ij}^m \leq F_my_{ij}^m \quad \forall \, (i,j)\in E, \, m \in M, \label{eq:9} & \\
&\sum_{(i,j) \in E} f_{ij}y_{ij}^m \leq \delta_m \quad \forall \, m \in M,\label{eq:14a}& \\
& y_{ij}^m \in \{0,1\}, x_{ij}^m \geq 0 \quad \forall \, (i,j) \in E, m \in M, \label{eq:11b}  & \\
& z_r^m \in \{0,1\}\quad \forall \, r\in R\setminus \{r_0\}, m \in M. \label{eq:13} &  
\end{flalign}

In the above formulation, constraints \eqref{eq:1} ensure that cardinality of in-degree and out-degree matches for each UV in a refueling station. Constraints \eqref{eq:2} ensure that all the UVs are used for the mission, i.e, each UV leave and return to the base station. Constraints \eqref{eq:3} help a feasible solution to stay connected, and constraints \eqref{eq:4} are indicator type where it forces $z^m_r$ to take a value of $1$ if an $m^{th}$ UV uses the refueling station $r \in R\setminus\{r_0\}$. Constraints \eqref{eq:5} state that each POI can be visited only once. Constraints \eqref{eq:7} and \eqref{eq:11} eliminate sub-tours starting from refueling stations $r\in R\setminus{r_0}$, and also defines the distance for the continuous variables $x_{ij}^m$ for each edge $(i,j) \in E$ and UV $m$. Constraints \eqref{eq:9} ensure that the fuel consumed by any UV $m$ to reach a base station does not exceed its fuel capacity $F_m$. Constraints \eqref{eq:14a} define the maximum time or distance allowed for each UV during a mission. In  constraints \eqref{eq:14a}, $\delta_m$ is the maximum distance allowed for each UV $m$. Finally, constraints \eqref{eq:11b} and \eqref{eq:13} impose the restrictions on the decision variables. The second-stage recourse problem for a scenario $\om$ and given first-stage decisions $x$ and $y$ is given by:

\begin{flalign}
& \text{Maximize} \quad \quad \phi^{ec}(x,y,\omega) = \quad -\sum_{(i,j)\in E, m \in M} e_{j}^m v_{ij}^m(\om) \notag& \\
&\text{Subject to: } \notag & \\
& \sum_{j\in V}f_{ij}v_{ij}^m(\om) = \sum_{j\in V}x_{ij}^m - \sum_{j\in V}x_{ji}^m - \alpha_{m}(\om)\sum_{j\in V} f_{ij} y_{ij}^m \forall i \in P, m \in M, \label{eq:14} & \\
& v_{ij}^m(\om) \leq y_{ij}^m \quad \forall (i,j) \in E, m \in M, \label{eq:15a} & \\
&f_{ri}v_{ri}^m(\om) = x_{ri}^m - \alpha_{h}(\om).f_{ri}y_{ri}^m \quad \forall \, i\in V, \, r \in R, \forall m \in M, \label{eq:11a} & \\
& v_{ij}^m(\om) \in \{0,1\} \quad \forall \, (i,j) \in E, m \in M.\label{eq:16} & 
\end{flalign}

Variables $v_{ij}^m(\om)$ maintain the feasibility of the constraints \eqref{eq:7}-\eqref{eq:11} for the given first-stage values $x$ and $y$ using the constraint \eqref{eq:14}. Constraints \eqref{eq:15a} state the dependence of $y_{ij}^m$ and $v_{ij}^m(\om)$. Finally, binary restrictions for $v_{ij}^m(\om)$ are presented in \eqref{eq:16}. Let the relaxed recourse problem for $\phi^{ec}(x,y,\omega)$ be represented as $\phi^{ec}_r(x,y,\omega)$. For the relaxed problem  $\phi^{ec}_r(x,y,\omega)$, the constraints \eqref{eq:16} are replaced by $0 \leq v_{ij}^m(\om) \leq 1$.

\begin{thm}\label{31}
	The objective values of $\phi^{ec}(x,y,\omega)$ and $\phi^{ec}_r(x,y,\omega)$ are same.
\end{thm}
\begin{proof}
	We need to show that the values of $v_{ijh}(\om)$ will be either $0$ or $1$ for $\phi^{ec}_r(x,y,\omega)$. Due to constraints  \eqref{eq:15a}, it is sufficient to show that any $v_{ij}^m(\om) > 0$ will be equal to 1. If $\alpha_{m}(\om)=1$, then $v_{ij}^m(\om) =0$ due to constraint \eqref{eq:7}. When $\alpha_{m}(\om)=0$, $\sum_{j\in V} f_{ij} y_{ij}^m = \sum_{j\in V}f_{ij}v_{ij}^m(\om)$ due to constraint \eqref{eq:7}, and let $\Gamma$ be a set with $(i,j,m) \in \Gamma$ for any $y_{ij}^m = 1$, and let $\Gamma'$ be the corresponding set for $v_{ij}^m(\om)$. By the constraints \eqref{eq:15a}, $\left\vert{\Gamma}\right\vert = \left\vert{\Gamma'}\right\vert$, so to satisfy $\sum_{j\in V} f_{ij} y_{ij}^m = \sum_{j\in V}f_{ij}v_{ij}^m(\om)$,  $v_{ij}^m(\om) = 1$ for any $(i,j,m) \in \Gamma'$. 
\end{proof}

The significance of the above theorem is presented in the decomposition algorithm section. Some more propositions are used to tighten the relaxation for two-stage formulation. Following proposition is used to remove the binary restrictions for the variables $z_{r}^{m}$, and also provides a better LP relaxation compared to the constraint \eqref{eq:4}. 

\begin{proposition} \label{prop:Con_Replacement}
	Binary restrictions on constraints \eqref{eq:13} are relaxed. Constraints \eqref{eq:4} and \eqref{eq:13} are replaced by the following:
	\begin{alignat}{3}
	& {y_{ri}^{m}} \leq z_{r}^{m} \hspace{0.5cm} \forall i \in P \cup  \{r_{0}\}, \  r\in \bar{R},m \in M,  \nonumber &\\  
	& 0\leq z_{r}^{m}\leq 1 \hspace{0.5cm} \forall  r \in \bar{R}, m \in M.  \nonumber &
	\end{alignat}
\end{proposition}

\begin{proof}
	See \cite{sundar2016exact} for the proof.
\end{proof}

\subsection{Tightening the Two-Stage Stochastic Formulation} \label{subsec:strengthen}
A constraint is strengthened if it eliminates fractional solutions to the two-stage stochastic formulation SVA-TS without removing any feasible integer solutions. The following proposition strengthens the inequalities \eqref{eq:9}.

\begin{prop} \label{prop:strengthen} Whenever triangle inequalities are maintained, i.e, for any $i,j,k \in V$, if $f_{ij} + f_{jk} \geqslant  f_{ki}$, then the constraints \eqref{eq:9} are strengthened using the following constraints:
	\begin{subequations}
		\begin{flalign}
		& x_{ij}^m \leqslant (F_m - t_j)y_{ij}^m \quad \forall j\in P,\, (i,j) \in E, m \in M, \nonumber & \\
		& x_{ir}^m \leqslant F_my_{ir}^m \quad \forall i \in V, r \in R, m \in M, \nonumber & \\
		& x_{ij}^m \geqslant (s_i + f_{ij}) y_{ij}^m \quad \forall i\in P, \, (i,j) \in E,  m \in M, \nonumber  &
		\end{flalign}
		\label{eq:strengthen}
	\end{subequations}
	where $t_i = \min_{r\in R} f_{ir}$ and $s_i = \min_{r\in R} f_{ri}$. 
\end{prop}
\begin{proof} Refer \cite{venkatachalam2018two}.
\end{proof}

\section{Algorithm}\label{alog}

The deterministic equivalent problem of a two-stage stochastic programming formulation is a large scale mixed-integer linear programming model. The branch-and-cut method in any commercial solver can be used, however it will be computationally challenging. Hence, we need an efficient decomposition algorithm to solve the instances of the stochastic integer problem described in the previous section. Another challenge in the two-stage SVA-TS formulation is the sub-tour elimination constraint \eqref{eq:3} in the first-stage problem. The total number of such constraints will be exponential, hence it is not practical to explicitly add the constraints to the two-stage model. Hence, these constraints have to be dynamically generated based on their need and fed to the solver whenever required. Hence, we present a decomposition algorithm and a methodology used to dynamically add sub-tour elimination constraints in this section.

\subsection{Decomposition Algorithm}\label{decomp}

The decomposition algorithm used for `SVA-TS' is a variant of L-shaped algorithm \cite{van1969shaped} where the first-stage solution is obtained by solving the first-stage problem with the deterministic parameters. Subsequently, the second-stage problems are solved based on the first-stage solutions and realizations of the random variable. Whenever, the first- and second-stage solutions are not optimal, optimality cuts are generated  based on the dual information of the second-stage problems and added to the first-stage problem. The optimality cuts approximate the value function of the second-stage profit in the first-stage problem. Thus, the dual information of the second-stage recourse problems are used to generate the optimality cuts for the first-stage, and the optimality cuts approximate the second-stage objective function. The process is iterated till we attain optimal solutions for the first- and second-stages. Since, the dual information for the second-stage problems are required, the second-stage problems are required to be linear programs, hence the importance of theorem \eqref{31}. Thus, the theorem helps to use a variant of  L-shaped method to solve the instances of SVA-TS. Otherwise, the binary restrictions for some of second-stage variables will make the value function as non-convex and lower semi-continuous, and a direct use of L-shaped method is not possible. The first-stage problem  is a mixed-integer program due to the binary restrictions for $y$ and $z$, and using theorem \eqref{31}, the second-stage problems are solved as linear programs. The information is iteratively passed between the stages till an optimal solution is attained for SVA-TS. The optimality cuts in the first-stage can be one single cut, where the dual information from all the second-stages are aggregated or a multi-cut, where dual information from each scenario will be represented as a cut in the first-stage. A multi-cut approach presents more detail than single-cut, however it add additional stress to the first-stage problem due to the volume of individual constraints added to the first-stage model. Since, there are binary variables and sub-tour elimination constraints \eqref{eq:3} in the first-stage problem, we adopted a single-cut approach for the computational experiments.

\subsubsection{Problem Reformulation}\label{decomp-model}	
The formulation SVA-TS given in \eqref{eq:1}-\eqref{eq:16} is divided into two problems, a master and the second-stage problem. Constraints \eqref{eq:1}-\eqref{eq:13} are reformulated as master problem along with an unrestricted variable $\theta$ to approximate the objective function value of the second-stage problem. The unrestricted variable is bounded using the optimality cuts generated using the dual information of the second-stage problems. The master problem `SVA-TS-MP' is given as follows:

\begin{subequations}
	\begin{flalign}
	& z^ k = \text{Maximize} \sum\limits_{\substack{(i,j) \in E: j \in P \\ m \in M}} e_{j}^m y_{ij}^m + \theta \label{eq:mp} & \\		
	& \text{ Subject to:} \nonumber & \\
	&  \eqref{eq:1} - \eqref{eq:13}, \nonumber  &\\  
	& \sum \limits_{\substack{\omega \in \Omega}} \sum\limits_{\substack{(i,j) \in E, \\ m \in M}} ((\pi_1(\omega)^{t \top}T_1)+(\pi_2(\omega)^{t \top}T_2) + (\pi_3(\omega)^{t \top}T_3)) y_{ij}^m + ((\pi_1(\omega)^{t \top}S_1) +(\pi_3(\omega)^{t \top}S_3)) x_{ij}^m + \theta \leq 0 \nonumber &\\ 
	& \qquad \qquad \qquad \qquad \qquad \qquad \qquad \qquad \qquad \qquad \qquad \qquad \qquad \qquad  \qquad \qquad  \qquad \qquad \qquad \ t \in \Pi, \label{eq-master-1a} &\\
	&\,\,\,\, \theta \in \mathbb{R}. \label{eq:bin_12} &
	\end{flalign}
	\label{eq:1MP}
\end{subequations}

In the master problem \eqref{eq:1MP}, $\pi_1(\omega)$, $\pi_2(\omega)$, and $\pi_3(\omega)$ are the dual vectors of the constraints \eqref{eq:14}, \eqref{eq:15a}, and \eqref{eq:11a}, respectively, for a scenario $\omega$. Similarly, $T_1$, $T_2$, and $T_3$ represent the co-efficient matrices for the variables $y_{ij}^m$ in the constraints \eqref{eq:14}, \eqref{eq:15a}, and \eqref{eq:11a}, respectively. Also, $S_1$ and $S_3$ represent the co-efficient matrices for the variables $x_{ij}^m$ in the constraints \eqref{eq:14} and \eqref{eq:11a}, respectively. Finally, $\theta$ is an unrestricted decision variable. Constraints \eqref{eq-master-1a} are the \textit{optimality} cuts, which are computed based on the optimal dual solution of second-stage problem $\phi^{ec}_r(y,x,\omega)$. Optimality cuts approximate the value function of the second-stage problems $\phi^{ec}_r(y,x,\omega)$. The details of the algorithm are summarized in Fig. \ref{figDA}. It should be noted that the two-stage model has relatively complete recourse property, i.e, $\phi^{ec}_r(y,x,\omega) < \infty$ for any $y_{ij}^m$ and $x_{ij}^m$. Hence, feasibility cuts are not required for the master problem.

\begin{figure}[!htb]
	\begin{hanglist}
		\item \textbf{Decomposition Algorithm}
		\item \hrulefill
		
		{\it\bf \item Step 0. } \textit{Initialize:} $n \leftarrow 0$, $lb \leftarrow -\infty$, $ub\leftarrow \infty$, and $\epsilon > 0$ is a user defined parameter and  $x^0, y^0, z^0$ are obtained as follows: argmin$ \{\sum_{(i,j)\in E, m \in M} e_{j}^m y_{ij}^m|\eqref{eq:1}-\eqref{eq:13}\}.$
		
		{\it\bf \item Step 1.} \textit{Solve second-stage problems:} Solve $\phi^{ec}_r(y,x,\omega)$ for each $\omega \in \Omega$, and obtain dual values $\pi_1(\omega)$, $\pi_2(\omega)$ and $\pi_3(\omega)$ for each second-stage problem. 
		
		{\it\bf  \item Step 2.} \textit{Optimality cut:} 
		For any integer solution for `SVA-TS-MP': Based on the dual values from the second-stage problem, generate an optimality cut (\ref{eq:1MP}b), and if it is violated then add it to the set $\Pi$ of the master problem `SVA-TS-MP.'
		
		{\it\bf  \item Step 3.} \textit{Solve master problem:} Solve the master problem SVA-TS-MP along with the new optimality cut, and let the objective function value be $u^n$. Set $ub \leftarrow min\{u^n, ub\}$.  Check for strongly connected components, and if a constraint \eqref{eq:3} is violated then Step 4 otherwise Step 5.					
		
		{\it\bf  \item Step 4.} \textit{Add sub-tour elimination constraints:} Add the corresponding infeasible constraint \eqref{eq:3}. Go to Step 3.					
		
		{\it\bf  \item Step 5.} \textit{Update bounds:} $v^n \leftarrow \{\sum_{(i,j)\in E, m \in M} e_{j}^m y_{ij}^m|\eqref{eq:1}-\eqref{eq:13}\} + \mathbb{E}_{\tilde{\omega}} \phi^{ec}_r(y,x,\omega)$, and set $lb \leftarrow max\{v^n, lb\}$. If $lb$ is updated, set incumbent solution to $y^{*} \leftarrow y^n, z^{*}\leftarrow z^n$ and $x^{*}\leftarrow x^n$.

		{\it\bf  \item Step 6.} \textit{Termination:} If $ub-lb < \epsilon|ub|$ then stop, $y^{*}, z^{*}$ and $x^{*}$ are $\epsilon$-optimal solutions, else set $n \leftarrow n+1$ and return to Step 1.
		
		\item \hrulefill
	\end{hanglist}
	\caption{Details of the branch and cut decomposition algorithm}
	\label{figDA}
\end{figure}

\subsection{Dynamic Sub-Tour Elimination Constraints}\label{det-model}

The sub-tour elimination constraints \eqref{eq:3} are exponential in number, hence it is computationally not efficient to add them explicitly. Hence, the constraints are relaxed from the master problem in the decomposition algorithm given in Fig. \ref{figDA}. During the branch-and-cut procedure used to solve the master problem, every feasible solution is checked whether any of the constraints \eqref{eq:3} are violated. If so, we add the corresponding infeasible constraints to the master problem.

The details of the algorithm used to find an infeasible constraint for a given integer feasible solution for the master problem \eqref{eq:3} are as follows.  A violated constraint \eqref{eq:3} can be described by a subset of vertices $S \subset V\setminus\{r_0\}$ such that $S\cap R \neq \emptyset$ and $y(S) = |S|$ for every $r \in S\cap R$. We find the strongly connected components (SCC) of $S$. Every SCC that does not contain the refueling station is a subset $S$ of $V\setminus\{r_0\}$ which violates the constraint \eqref{eq:3}. We add a sub-tour elimination constraint for each SCC and continue solving the original problem. Many off-the-shelf commercial solvers provide a feature called ``solver callbacks'' to implement such an algorithm into its branch-and-cut framework.

\section{Computational Experiments}\label{comp}

\subsection{Algorithm Performance}\label{alg}
The  decomposition algorithm presented in Fig. \ref{figDA} was implemented in Java, and the dynamic sub-tour elimination constraints are implemented using solver callback functionality of CPLEX version 12.9 \cite{CPLEX12}. All the computational runs were performed on a Dell Precision T5500 workstation (Intel Xeon E5630 processor @2.53 GHz, 12 GB RAM). A time limit of one hour was used for each of the instances, and the computational runtimes are reported in seconds. Similar to the computational experiments in \cite{venkatachalam2018two} and \cite{sundar2017analysis}, the performance of the algorithm was tested with randomly generated test instances. A square grid of [100, 100] was used to generate the random instances. There were four refueling stations, and the locations for all the refueling stations were fixed. In steps of ten, the number of POIs were varied from 10 to 60, and the locations of the POIs were randomly generated within the square grid. For each $|P| \in \{10,20,30,40,50,60\}$, we generated five random instances. Two UVs were used for the computational experiments, and the fuel capacity of the UVs was $F_m$. The parameter $F_m$ was varied linearly with a parameter $\lambda$. The parameter $\lambda$ represents the maximum distance between any pair of nodes within the instance. The fuel capacity $F_m$ is chosen based on four different combinations for $\lambda$ given as $\{2.25 \lambda, 2.5 \lambda, 2.75 \lambda, 3\lambda \}$. The parameter $f$ representing the distance traveled or time consumed between a pair of nodes is the Euclidean distance between the pair. There are about 240 random instances for the computational experiments.

\indent Using the generated instances, three different computational runs were performed. The first experiment was the implementation of the algorithm given in Fig. \ref{figDA}, and the second set of experiments were conducted using the deterministic equivalent formulation (DEP). The DEP formulation is the entire model with constraints \eqref{eq:1}-\eqref{eq:16}. The third set of experiments were conducted using the linear relation for the second-stage, i.e., $\phi^{ec}(y,x,\omega)$ is replaced by $\phi^{ec}_r(y,x,\omega)$.  There are two UVs, and we constructed four different types of scenarios by changing the availability of the second UV as 100\%, 75\%, 25\% and 0\%. The profits for UV1 and UV2 are derived using the uniform distributions $\mathcal{U}(0, 150)$ and $\mathcal{U}(0, 170)$, respectively. Additionally, to reflect heterogeneity among the UVs, UV1's incentive for 50\% of the POIs is kept at zero, i.e, UV2 is better equipped and has higher incentives compared to UV1, however its availability is uncertain. 

\indent Fig. \ref{rFig} presents box-plot for the runtime performance of two-stage formulation, reformulation and decomposition algorithm denoted as a, b, and c, respectively. The MIP Gap represents the difference between the upper and lower bounds of the objective function, and a stipulated run time of one hour was allowed for each instance. A lower MIP gap represents a better quality solution, and is closer to optimality. The runs using reformulation and decomposition algorithm performed well compared to the two-stage model, however decomposition algorithm had the  lowest dispersion for runtimes. Some of the smaller instances were quicker using DEP formulation, however they had the largest dispersion compared to reformulation or decomposition algorithm. Especially, for the instances with 60 POIs, reformation and decomposition algorithms performed better than DEP. The performance of reformulation and decomposition were closer for 60 instances, however for some of the runs, decomposition algorithm was better than reformulation. Table \ref{rtime} and \ref{runtime} present the run time characteristics where `BD' and `SEC' represent number of optimality and sub-tour elimination cuts, respectively. The model was able to solve upto 30 POIs optimally, and about 90\% of instances with 40 POIs. However, for instances with 50 and 60 POIs, the model was able to solve about 30\% of the instances to optimality. As denoted in Fig. \ref{rFig}, the reformulation and decomposition algorithm had distinct advantage over the DEP formulation when the number of POIs exceeded 30. \\

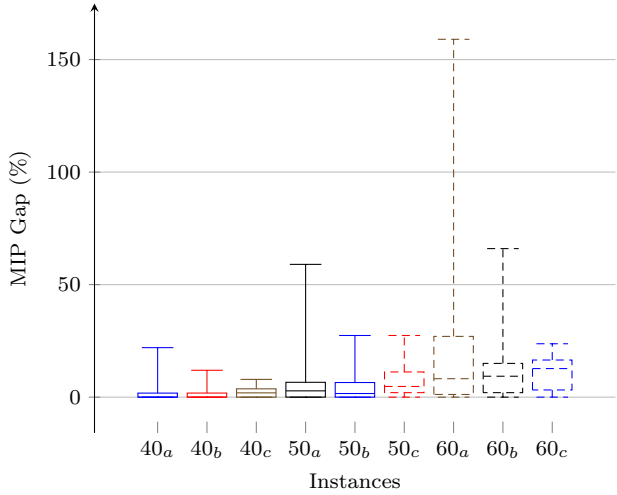
\begin{figure}	
	\begin{tikzpicture}
	\begin{axis}[
	boxplot/draw direction=y,
	x axis line style={opacity=0},
	axis x line*=bottom,
	axis y line=left,
	enlarge y limits,
	ymajorgrids,
	xtick={1,2,3,4,5,6,7,8,9},
	xticklabels={$40_a$,$40_b$,$40_c$,$50_a$,$50_b$,$50_c$,$60_a$,$60_b$,$60_c$},
	xlabel={Instances},
	ylabel={MIP Gap (\%)},
	]
	\addplot+ [
	boxplot prepared={
		lower whisker=0, lower quartile=0,
		median=0,
		upper quartile=1.8, upper whisker=22,
	},
	] coordinates {};
	\addplot+ [
	boxplot prepared={
		lower whisker=0, lower quartile=0,
		median=0,
		upper quartile=1.82, upper whisker=11.95,
	},
	] coordinates {};
	\addplot+[
	boxplot prepared={
		lower whisker=0, lower quartile=0,
		median=1.9,
		upper quartile=3.7, upper whisker=7.9,
	},
	] coordinates {};
	\addplot+[
	boxplot prepared={
		lower whisker=0, lower quartile=0,
		median=2.8,
		upper quartile=6.6, upper whisker=59,
	},
	] coordinates {};
	\addplot+[
	boxplot prepared={
		lower whisker=0, lower quartile=0,
		median=1.62,
		upper quartile=6.5, upper whisker=27.4,
	},
	] coordinates {};
	\addplot+[
	boxplot prepared={
		lower whisker=0, lower quartile=2,
		median=4.75,
		upper quartile=11.2, upper whisker=27.4,
	},
	] coordinates {};
	\addplot+[
	boxplot prepared={
		lower whisker=0, lower quartile=1.2,
		median=8.2,
		upper quartile=27, upper whisker=159,
	},
	] coordinates {};
	\addplot+[
	boxplot prepared={
		lower whisker=0, lower quartile=2,
		median=9.3,
		upper quartile=15, upper whisker=66,
	},
	] coordinates {};
	\addplot+[
	boxplot prepared={
		lower whisker=0, lower quartile=3.2,
		median=12.7,
		upper quartile=16.5, upper whisker=23.73,
	},
	] coordinates {};
	\end{axis}
	\end{tikzpicture}
	\caption{Box-plots to denote the algorithm performance for formulation (a), reformulation (b), and decomposition algorithm (c)}
	\label{rFig}
\end{figure}

\begin{table}[H]
	\centering	
	\begin{tabular}{|c|c|c|c|}
		\toprule 
		\bf  \# POIs & \bf \# BD &	\bf\# SEC & \bf Avg. MIP Gap(\%) \\ 
		\midrule 		
		10 & 4	& 84 &	0 \\
		20 & 11 & 740 &	0 \\
		30 & 30 & 5,536 &	0 \\
		40 & 35 & 15,935 &	2.35 \\
		50 & 42 & 18,560 &	7.80 \\
		60 & 47 & 23,000 &	11.15 \\
		\hline 		
	\end{tabular}
	\caption{Number of optimality and sub-tour elimination cuts, and average MIP gap}
	\label{rtime}
\end{table}

\begin{table}[H]
	\centering	
	\begin{tabular}{|c|c|c|c|c|c|}
		\toprule 
		\bf  \# POIs & Min &	Max & Avg & Med & Std \\ 
		\midrule 		
		10 & $<$1 & 7 & $<$1 & $<$1 & $<$1 \\
		20 & $<$1 & 240 & 27 & 3 & 54 \\
		30 & $<$1 & 3,600 & 1,136 & 175 & 1,494 \\
		40 & 30 & 3,600 & 2,640 & 3,600 & 1,518 \\
		50 & 380 & 3,600 & 3,439 & 3,600 & 720 \\
		60 & 3,600 & 3,600 & 3,600 & 3,600 & 0 \\
		\hline 		
	\end{tabular}
	\caption{Runtime (seconds) for decomposition algorithm}
	\label{runtime}
\end{table}

\begin{figure}
	\centering
	\begin{subfigure}{.5\textwidth}
		\centering
		\includegraphics[width=.77\linewidth]{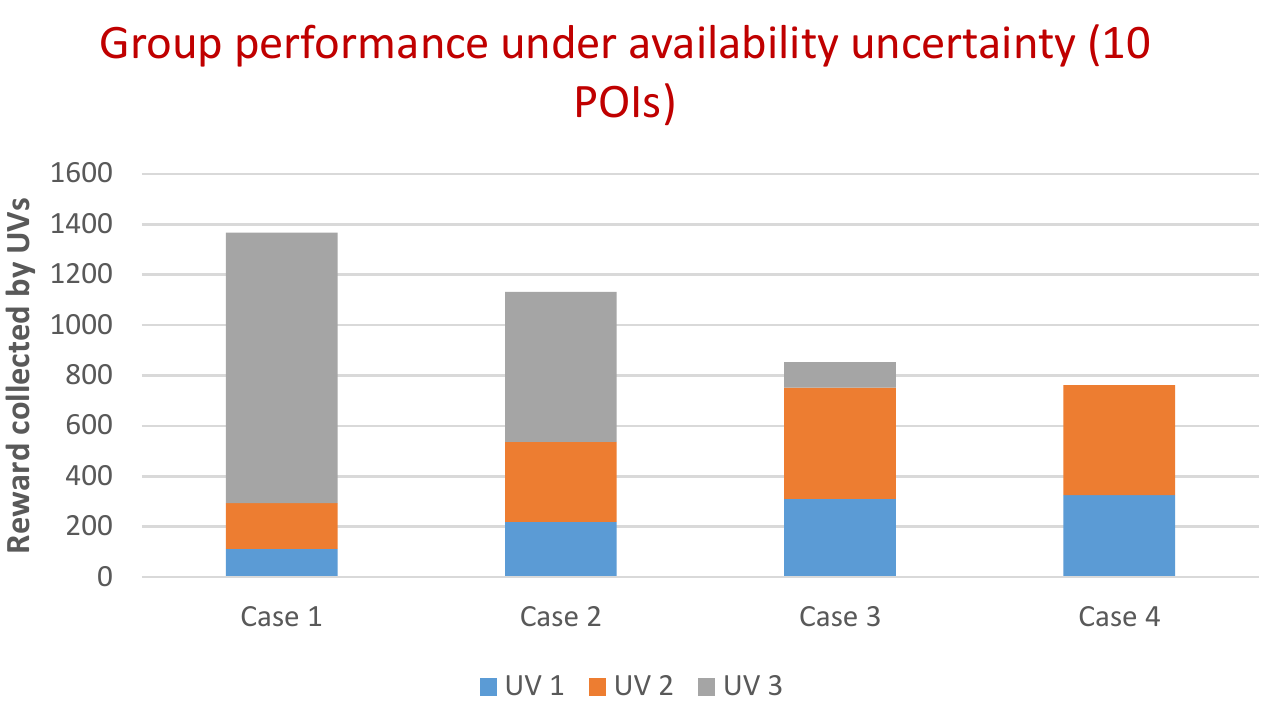}
		\caption{10 POIs - Three UVs, and UV 3 with limited availability}
		\label{fig:sub1}
	\end{subfigure}%
	\begin{subfigure}{.5\textwidth}
		\centering
		\includegraphics[width=.77\linewidth]{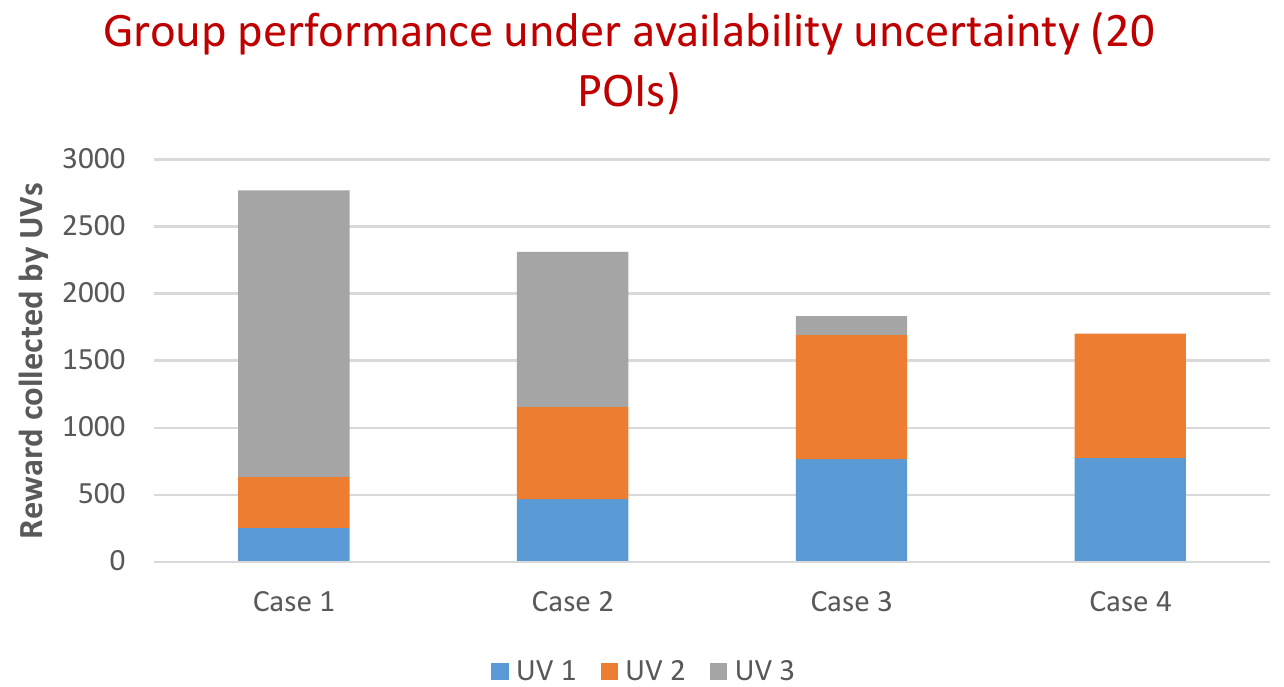}
		\caption{20 POIs - Three UVs, and UV 3 with limited availability}
		\label{fig:sub2}
	\end{subfigure}
	\caption{Distribution of rewards among UVs where 	Cases 1, 2, 3, and 4 represent the availability of UV3 as 100\%, 75\%, 25\%, and 0\%, respectively}
	\label{fig:qual}
\end{figure}

\indent Fig. \ref{fig:qual} quantifies the use of the proposed two-stage stochastic programming approach compared to a deterministic model. For this experiment setup, there are three UVs, and Figs. \ref{fig:sub1} and \ref{fig:sub2} use 10 and 20 POIs, respectively. Four different cases were evaluated with the availability of UV3 varied at 100\%, 75\%, 25\% and 0\%. Especially, UV1 and UV2 were able to reallocate the POIs during the absence of UV3 in the cases 2, 3, and 4, so the overall rewards collected is maximized. This is an evaluation of `value of stochastic solution' \cite{birge1982value} for the two-stage stochastic programming model, and the benefits are in between 10\%-30\% by using the proposed two-stage stochastic model compared to the deterministic model.

\subsection{ROS Simulation}\label{sim}

\indent In a data-driven simulation study, we quantified the impact of the two-stage stochastic model in rviz \cite{hershberger6rviz} 3D-simulation environment for ROS \cite{ros}. The overall architecture of the implementation is shown in Fig. \ref{rosov}. Turtlebots \cite{tbot} were used within the turtlebot\_stage package environment as shown in Fig. \ref{ros}. The blue ovals are recharging stations and the red rectangles are the POIs. Each Turtlebot visits the POIs in a sequence determined by the two-stage stochastic or deterministic model.  The POIs are randomly selected within the turtlebot\_stage package environment, and the recharging stations are selected a priori in five different locations as shown in Fig. \ref{ros}. The ROS navigation stack \cite{rosnav} is used to calculate the time between each pair of recharging stations and POIs for the two-stage stochastic programming model $(i,j) \in E$ denoted as $f'_{ij}$. However, they cannot be directly used as $f_{ij}$ in the two-stage stochastic model since many pairs of $f'_{ij}$ do not concur with the triangle inequality property. Hence, a conversion for $f'_{ij}$ to $f_{ij}$ is performed using Dijkstra's algorithm \cite{dijkstra1959note}. Any shortest path planning between a pair of nodes will provide the appropriate mapping between $f'_{ij}$ to $f_{ij}$, so the triangle inequality property can be preserved. We obtained two sets of missions from deterministic and two-stage stochastic programming models, and the realization of profits are measured quantitatively. Based on the scenarios, the probability of failure is chosen and the decision is made whether to proceed to the next POI or stop due to failure. For example, with the scenario of 25\% failure for an UV,  a random number is chosen between 0 and 1, and when the number is less than 0.25, the UV is declared as a failure, and the mission for that UV is terminated with the incentives collected so far.

\begin{figure}[h!]
	\centering
	\includegraphics[scale=0.40]{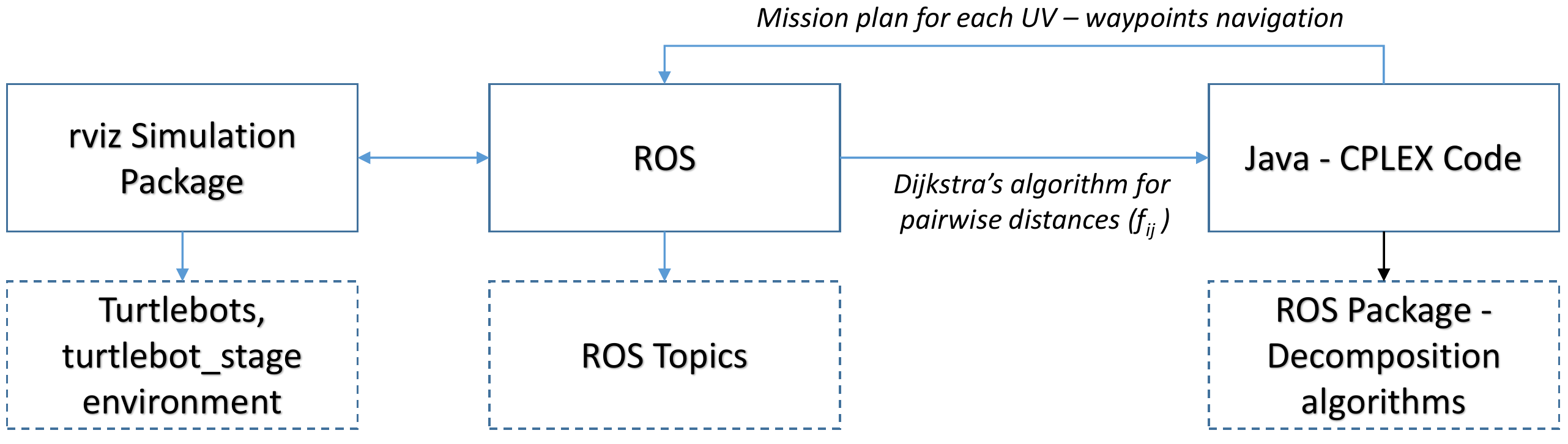}
	\caption{Software framework consisting of interaction among ROS, rviz, Turtlebots, and decomposition algorithms}
	\label{rosov}
\end{figure}

\indent The experiments were conducted using ROS Kinetic Kame \cite{kame} using a Ubuntu 16.04 (Xenial) release and the ROS topic and publisher were developed using Python programming language. Similar to the previous section, two Turtlebots (named as Turtlebot1 and Turtlebot2) were used. Four different types of scenarios were constructed by changing the availability of the second Turtlebot as 100\%, 75\%, 25\% and 0\%, and the scenarios are named as S1, S2, S3, and S4 in the figures, respectively. Four different instances of maps were used to conduct the experiments. Results of the experiments are depicted in Figs. \ref{exp1} and \ref{exp2}. Fig. \ref{exp1} exhibits the transfer of incentives from Turtlebot1 to Turtlebot2 for different scenarios. For S1, Turtlebot2 collected highest incentive and this got reversed with S4. The take away here is that the algorithm is working appropriately and transferring POIs based on availability. 

\begin{figure}[h!]
	\centering
	\includegraphics[scale=0.30]{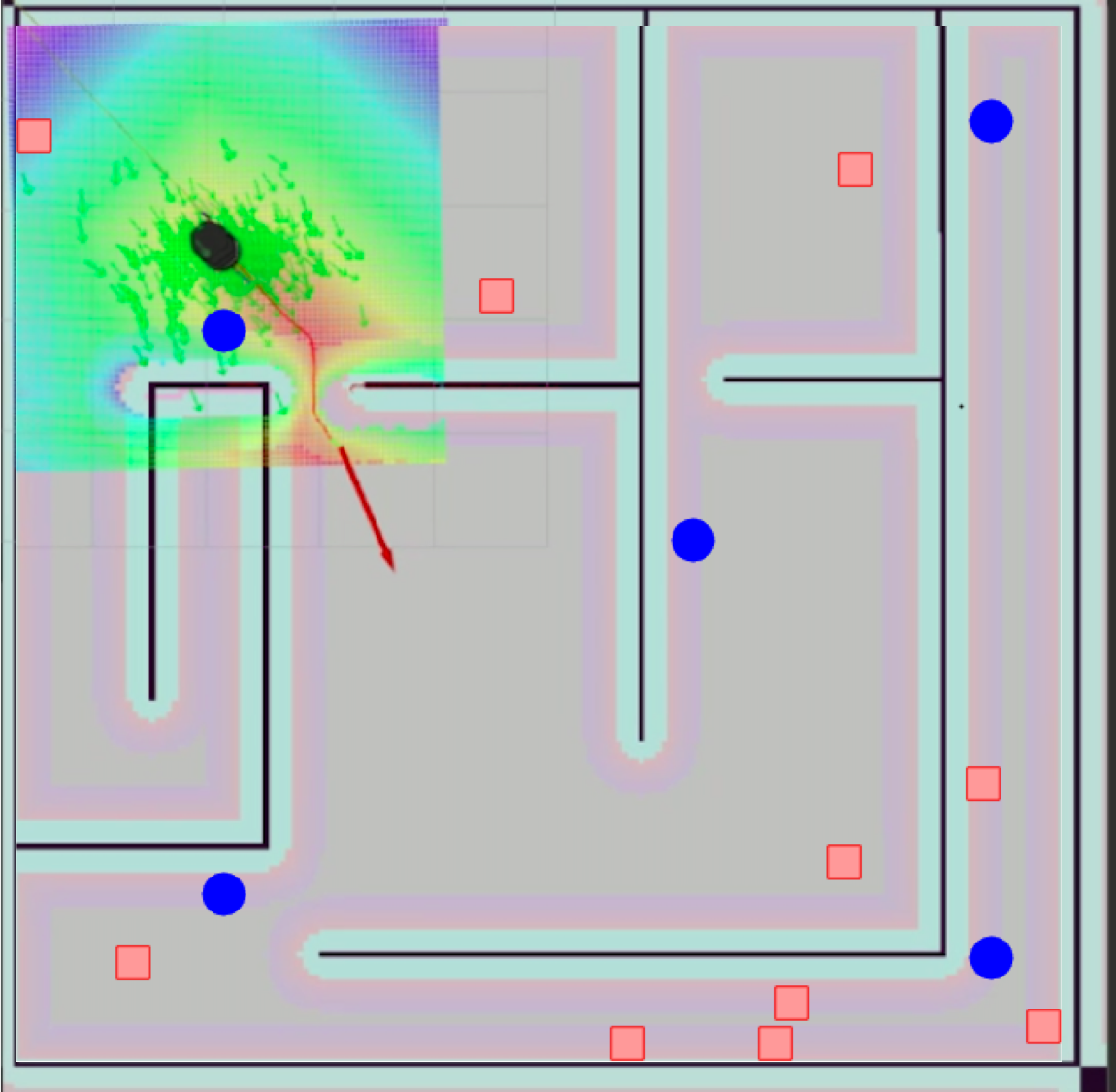}
	\caption{Screenshot of Rviz navigation using Turtlebot for one of the four instances. The blue ovals are recharging stations and the red squares are POIs, and the black oval is a Turtlebot}
	\label{ros}
\end{figure}

\begin{figure}[h!]
	\centering
	\includegraphics[scale=0.45]{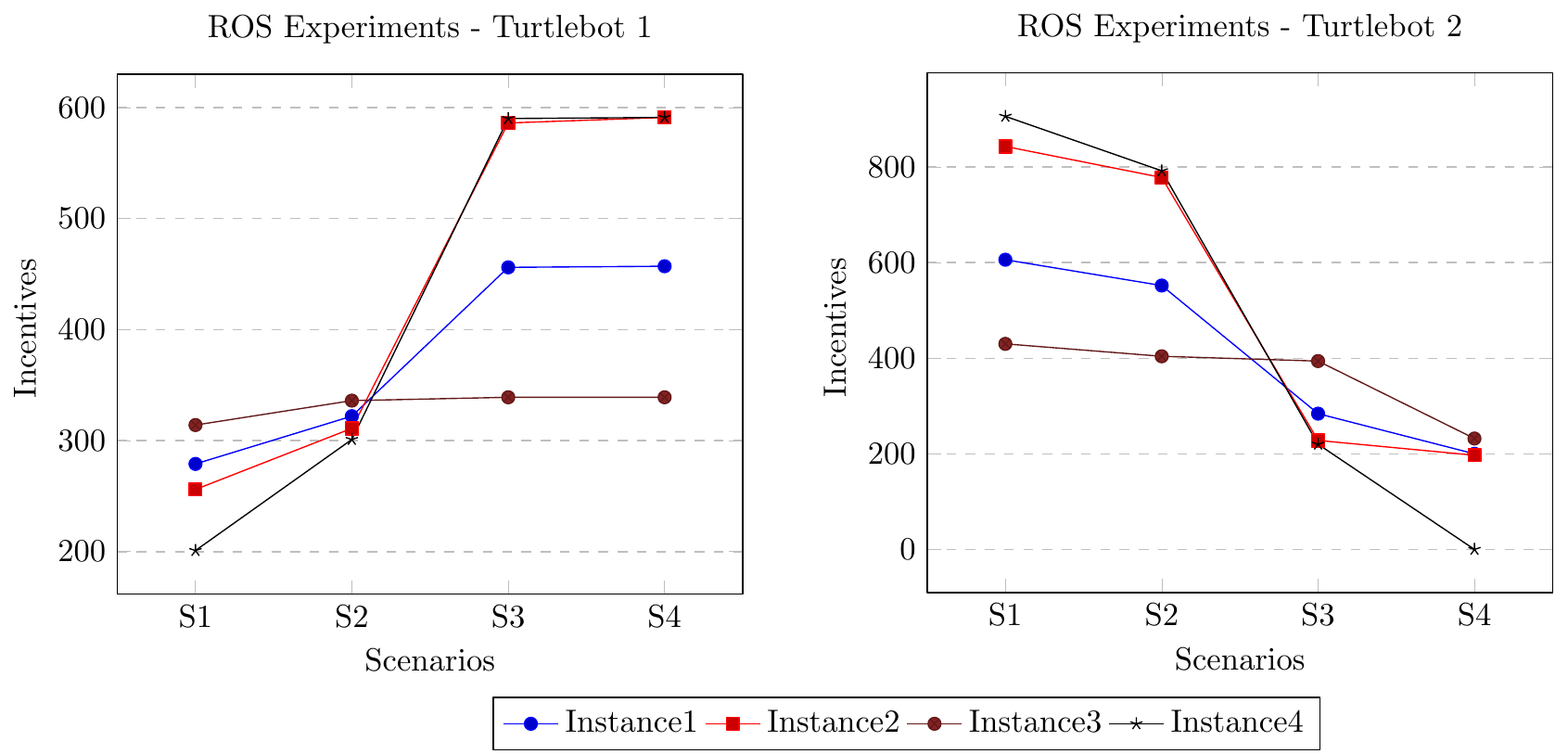}
	\caption{Results for four different scenarios. The incentives get shifted based on the  availability of the second Turtlebot}
	\label{exp1}
\end{figure}

\begin{figure}[h!]
	\centering
	\includegraphics[scale=0.45]{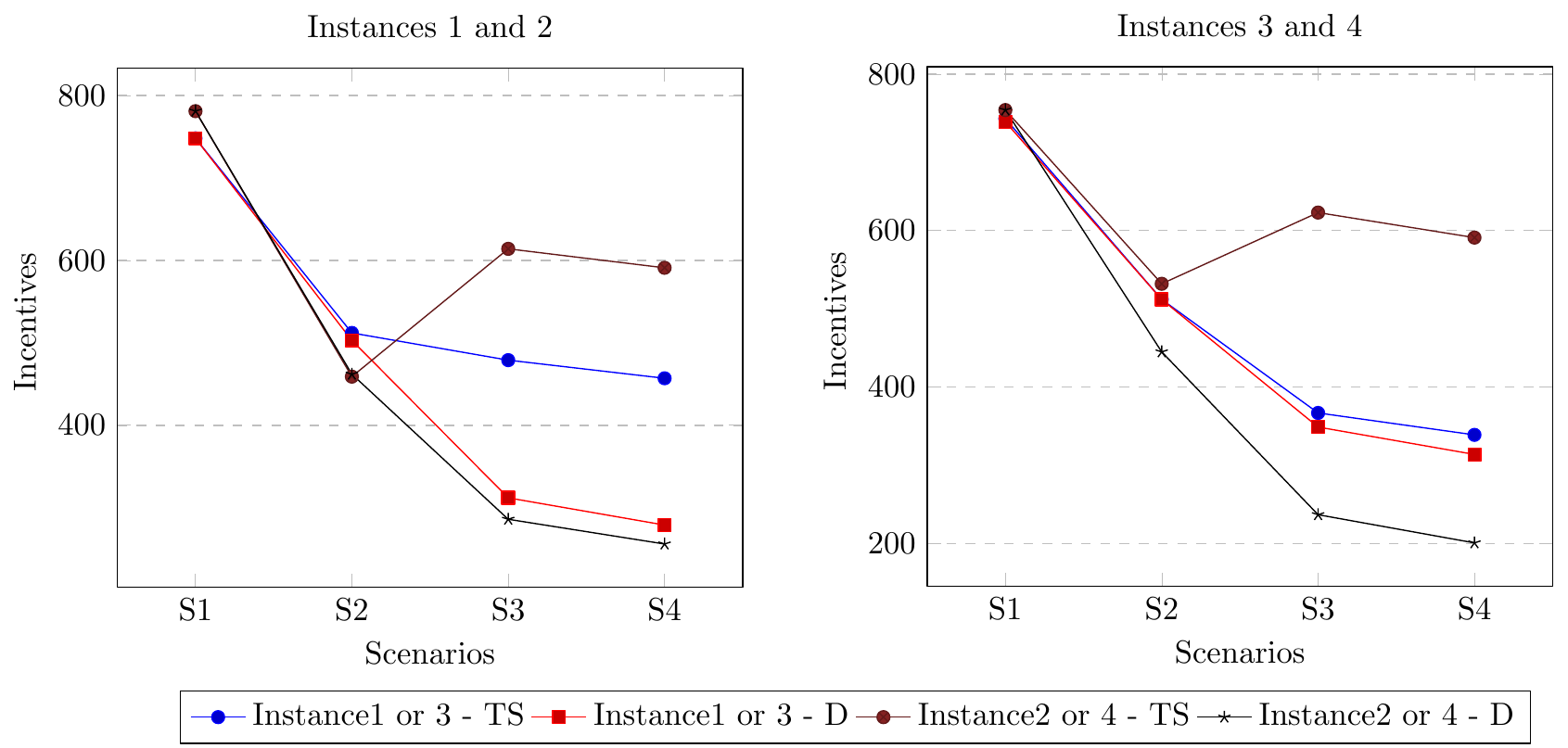}
	\caption{Comparison between the solutions using two-stage stochastic programming and deterministic models. `TS' and `D' represent two-stage and deterministic models, respectively}
	\label{exp2}
\end{figure}

\indent In the second set of experiments, we compared the results from two-stage model with a deterministic model. A deterministic model will have the objective function $ \sum_{\substack{(i,j)\in E \\ m \in M}} e_{j}^m y_{ij}^m$ and the constraints \eqref{eq:1}-\eqref{eq:13}. This is similar to estimating value of stochastic solution (VSS) \cite{birge1982value} in stochastic programming. In Fig. \ref{exp2}, `TS' and `D' represent solutions from two-stage and deterministic models, respectively. For example, the total incentives dropped from 457 units to 279 units for S4. On an average, the VSS for S2 is around 17\%, and for S3 and S4 is about 40\%. The take away here is that the stochastic model is adjusting based on the availability of the UVs to increase the incentives whereas the deterministic model results are a direct consequence of UV availability.

\section{Conclusion}\label{con}
\indent In this work, we address the fundamental issue of accounting for changes in availability of vehicles during the course of a mission, and present a two-stage stochastic programming model. The model determines routes for heterogeneous unmanned vehicles (UVs) in the presence of uncertainties in their availability as an effective means of planning UV missions to meet mission objectives despite the uncertainties. The model considers set of recharging stations, and the objective is to maximize the expected incentives collected by the UVs. Also, we present a reformulation so the model is amenable to decomposition algorithms. A modified version of the L-shaped algorithm is presented, and the value of reformulation and decomposition algorithm are evaluated using extensive computational experiments. Additionally, a ROS simulation environment is used to compare the plans suggested by two-stage model and deterministic model. This provides a `\textit{value of stochastic solution}' which are conventionally used to quantify the use of stochastic models instead of their deterministic counterpart. Future work will involve extending the two-stage model to consider other uncertainties like travel time or range. In terms of the decomposition algorithm, column generation approaches can be evaluated for the instances with higher number of UVs. Also, from a stochastic programming perspective, other risk-measures like conditional value at risk or expected excess can be considered whenever the uncertainties have a large dispersion.

\vspace{-3mm}

\begin{acknowledgements}
The authors like to acknowledge the technical and financial support of the Automotive Research Center (ARC) in accordance with Cooperative Agreement W56HZV-19-2-0001 U.S. Army CCDC Ground Vehicle Systems Center (GVSC) Warren, MI. Distribution A. Approved for public release; distribution is unlimited. OPSEC \# 4482.
\end{acknowledgements}

%
\section*{Conflict of interest}
The authors declare that they have no conflict of interest.

\bibliographystyle{spmpsci}      
\bibliography{Referencias,networks}   

\end{document}